%% file: ms.tex
\documentclass[CM,Submssn,SecEq,fleqn,reqno]{degruyter-crelle} %% Equations numbered as (1.1), (1.2) etc.

\pdfoutput=1

\firstpage{}\volume{}\volumeyear{}\copyrightyear{}\doiyear{}\doi{}\received{XXX}\revised{XXX}

\title{The Prime Number Formula of Gandhi}

\author{Berndt}{Gensel}{B. Gensel}{Spittal an der Drau}

\contact{Studiengang Bauingenieurwesen und Architektur, Fachhochschule K\"arnten, \\ Villacher Str. 1, A-9800, Spittal an der Drau, \"Osterreich
\\ www.gensel.at}{b.gensel@fh-kaernten.at}

\makeindex

\input{definitionen.tex}

\begin{document}

\input{Release.tex}

\input{KapAbstract.tex}
\input{KapIntroduction.tex}
\input{KapProof.tex}
\input{KapOrder.tex}

\input{Literatur.tex}

\end{document}

%% file: definitionen.tex
\theoremstyle{plain}
 \newtheorem{theorem}{Theorem}[]

\theoremstyle{definition}

\newcommand{\M}[1]{\mathbb{#1}}

\newcommand{\opn}[1]{\operatorname{#1}}

\newcommand{\pz}[1]{#1\sharp}

\newcommand{\indbox}[2]{
\scriptsize
\begin{array}{c}
#1 \\ #2
\end{array}
}
\newcommand{\ggT}{\opn{\scriptsize GCD}}
\newcommand{\CGD}{\opn{GCD}}

%% file: Release.tex
\begin{center} 
\small Release: \today 
\end{center}
\normalsize

%% file: KapAbstract.tex
\begin{abstract}
With the formula of Gandhi you can determine the on $p_n$ immedately subsequent prime $p_{n+1}$ from the knowledge of the primes $p_1, p_2, \dots , p_n$. An elementary proof of its trueness will be detailed shown in this paper. Finally the question for the order of the primes on the number line will be discussed.
\end{abstract}


%% file: KapIntroduction.tex
\section{Introduction}
1971 has J.W. \textbf{Gandhi} shown in \cite{1}  a formula to calculate theoretically the on $p_n$ immedately subsequent prime $p_{n+1}$ from the divisors of the \textit{primorial}  
\footnote{
$\pz{p_n} := \prod_{i=1}^n p_i$
}
$\pz{p_n}$.
\begin{equation}\label{Ga-1.1}
p_{n+1} = \left[1-\log_2\left(-\frac{1}{2}+\sum_{d\mid p_n\sharp}
\frac{\mu(d)}{2^d-1}\right)\right].
\footnote{We use the notation of Ribenboim (see \cite{3}, S. 140)}
\end{equation}
Practically the calculability is limited insofar as the term $2^d$ in the divisor of the summands reach very soon such values which cannot be numerically calculated.

An elementary proof for this formula came 1972 in \cite{4} from C. \textbf{Vanden Eynden}. S.W. \textbf{Golomb} has used the binary code of the number $1$ for his proof \cite{2} in 1974.

In (\ref{Ga-1.1}) $\mu(d)$ is the \textit{M\"obius function}
\[
\mu(d) = \left\lbrace 
\begin{array}{ll}
1 &\mbox{if } d=1 \\
(-1)^r &\mbox{if } d \mbox{ is a product of }r\mbox{ different prime factors}\\
0 &\mbox{if }d\mbox{ is not free of squares.}
\end{array}
\right.
\]

Let be
\begin{equation}\label{Ga-1.2}
\theta(n) := -\frac{1}{2}+\sum_{d \mid \pz{p_n}}\frac{\mu(d)}{2^d-1}
\end{equation}
Considering of $\log_2 2 = 1$  the formula (\ref{Ga-1.1}) get the form
\begin{equation}\label{Ga-1.3}
p_{n+1} = \left[\log_2 2 - \log_2 \theta(n) \right]
= \left[\log_2 \frac{2}{\theta(n)}\right].
\end{equation}

%% file: KapProof.tex
\section{Proof}
For the proof of the trueness of (\ref{Ga-1.3}) we need still three elementary theorems.

\begin{theorem}\label{S-5.1}
For every $a \in \M{N}$ holds:
\[
\sum_{k=1}^\infty 2^{-ka} = \frac{1}{2^a-1}.
\]
\end{theorem}
\begin{proof}
\begin{eqnarray*}
\sum_{k=1}^\infty 2^{-ka} &=&
\sum_{k=0}^\infty 2^{-ka} - 1
\mbox{ and as geometrical series} \\
 &=&\frac{1}{1-2^{-a}}-1 \\
 &=&\frac{2^{-a}}{1-2^{-a}}
 = \frac{1}{2^a-1}.
\end{eqnarray*}
\end{proof}

The proof idea for the following theorem comes from the proof of the formula of Gandhi from Vanden Eynden (\cite{4}) according to the book of RIBENBOIM (\cite{3}, S. 141/142).
\begin{theorem}\label{S-5.3}
\[
\sum_{d|p\sharp} \frac{\mu(d)}{2^d-1}=
\sum_{\ggT(t,p\sharp)=1} 2^{-t}.
\]
\end{theorem}
\begin{proof}
With theorem \ref{S-5.1} is
\[
\sum_{d|p\sharp} \frac{\mu(d)}{2^d-1}=
\sum_{d|p\sharp} \mu(d) \cdot \sum_{k=1}^{\infty} 2^{-kd}=
\sum_{k=1}^{\infty} \sum_{d|p\sharp} \mu(d) \cdot 2^{-kd}.
\]
In the dexter sum occur terms $\mu(d) \cdot 2^{-t}$ for $t \geq 1$ if and only if $d$ is a divisor of $\CGD(t,p\sharp)$. 
Therefore the coefficients of $2^{-t}$ are
\[
\sum_{d|\ggT(t,p\sharp)}\mu(d)
\]
and we can this sum also note as
\[
\sum_{t=1}^{\infty} 2^{-t} \sum_{d|\ggT(t,p\sharp)}
\mu(d).
\]
For the M\"obius function $\mu(d)$ and an integer $m \in \M{N}$ however holds (see \cite{3}, S. 141)
\begin{equation}\label{G-5.1}
\sum_{d|m} \mu(d) = \left\lbrace
\begin{array}{l}
1 \mbox{ if } m=1 \\
0 \mbox{ if } m>1.
\end{array}
\right.
\end{equation}
Hence the second sum becomes to zero for all $t$ with $\CGD(t,p\sharp)>1$. Only the summands with $\CGD(t,p\sharp)=1$ remain, therefore we get
\[
\sum_{d|p\sharp} \frac{\mu(d)}{2^d-1}=
\sum_{
\ggT(t,p\sharp)=1} 2^{-t}.
\] 
\end{proof}

\begin{theorem}\label{S-5.4}
\[
\sum_{t>n} 2^{-t} = 2^{-n}.
\]
\end{theorem}
\begin{proof}
The series $\sum_{k=0}^\infty 2^{-k}$ has as geometrical series the sum value
\[
\sum_{k=0}^\infty 2^{-k} = \frac{1}{1-2^{-1}} = 2.
\]
The $n$-th partial sum has the value
\[
\sum_{k=0}^n 2^{-k} = \frac{1-2^{-(n+1)}}{1-2^{-1}} = 2-2^{-n}.
\]
And therefore is
\[
\sum_{t>n} 2^{-t} = \sum_{k=0}^\infty 2^{-k}
- \sum_{k=0}^n 2^{-k} = 2 - 2 + 2^{-n} = 2^{-n}.
\]
\end{proof}
Additionally it's true for every prime $p$:
\[
\sum_{t>n} 2^{-t} = \sum_{\indbox{t>n}{\ggT(t,\pz{p})=1}} 2^{-t} +
\sum_{\indbox{t>n}{\ggT(t,\pz{p})>1}} 2^{-t}
\]
and therefore
\begin{equation}\label{G-5.1.0}
\sum_{\indbox{t>n}{\ggT(t,\pz{p})=1}} 2^{-t} < \sum_{t>n} 2^{-t}
= 2^{-n}.
\end{equation}
If we execute the theorem \ref{S-5.3} to our function $\theta(n)$, we get
\begin{eqnarray*}
\theta(n) &=&
-\frac{1}{2}+\sum_{d \mid \pz{p_n}} \frac{\mu(d)}{2^d-1}\\
&=& -\frac{1}{2} + \sum_{\ggT(t,\pz{p_n})=1} 2^{-t}\\ 
&=& \sum_{\indbox{t \geq 2}{\ggT(t,\pz{p_n})=1}} 2^{-t}.
\end{eqnarray*}

Because for $2 \leq t \leq p_n$ always is $\CGD(t,\pz{p_n})>1$ and $p_{n+1}$ is the least natural number which is prime to $\pz{p_n}$, we get
\begin{eqnarray}
\theta(n)&=& \sum_{\indbox{t \geq p_{n+1}}{\ggT(t,\pz{p_n})=1}} 2^{-t} \label{G-5.1.1}\\
\theta(n)&=& 2^{-p_{n+1}} +\sum_{\indbox{t > p_{n+1}}{\ggT(t,\pz{p_n})=1}} 2^{-t} \nonumber\\
&=& 2^{-p_{n+1}} + r_n \label{G-5.2}
\end{eqnarray}
with 
\begin{equation}\label{G-5.3}
r_n = \sum_{\indbox{t > p_{n+1}}{\ggT(t,\pz{p_n})=1}} 2^{-t}.
\end{equation}
And because of (\ref{G-5.1.0}) and theorem \ref{S-5.4} is
\[
r_n < \sum_{t>p_{n+1}} 2^{-t} = 2^{-p_{n+1}}.
\]
Therefore it follows from (\ref{G-5.2})
\begin{eqnarray}
&&2^{-p_{n+1}} < \theta(n) = 2^{-p_{n+1}} + r_n
\mbox{ and}\nonumber\\
&&2^{-p_{n+1}} < \theta(n) < 2 \cdot 2^{-p_{n+1}}
\mbox{ resp.}\label{G-5.4}\\
&&2^{p_{n+1}+1} > \frac{2}{\theta(n)} > 2^{p_{n+1}}
\mbox{ and}\nonumber\\
&&p_{n+1} < \log_2\frac{2}{\theta(n)} < p_{n+1}+1
\mbox{ and finally}\label{G-5.4.0}\\
&&p_{n+1} = \left[\log_2\frac{2}{\theta(n)}\right].\nonumber
\end{eqnarray}\qed

The proof is based on the fact that the prime $p_{n+1}$ is the least natural number $>1$ which is prime to $\pz{p_n}$:
\begin{equation}\label{G-5.4.1}
p_{n+1} = \min\left(t \in \M{N}\setminus\lbrace 1\rbrace \mid \ggT(t,\pz{p_n})=1\right).
\end{equation}
The upper bound of (\ref{G-5.4}) can be lessened.
(\ref{G-5.3}) can be written as
\[
r_n = 2^{-p_{n+2}}+\sum_{\indbox{t > p_{n+2}}{\ggT(t,\pz{p_n})=1}} 2^{-t},
\]
since $p_{n+2}$ is because of $p_{n+2} < 2p_{n+1}$ the least natural number which is prime to $\pz{p_n}$ and is greater than $p_{n+1}$.
With (\ref{G-5.1.0}) we get
\[
\sum_{\indbox{t > p_{n+2}}{\ggT(t,\pz{p_n})=1}} 2^{-t} < 2^{-p_{n+2}} 
\]
and therefore
\[
r_n < 2 \cdot 2^{-p_{n+2}} \leq 2 \cdot 2^{-p_{n+1}-2} = 2^{-p_{n+1}-1},
\]
because additionally $p_{n+2} \geq p_{n+1}+2$ holds.
With this the residual $r_n$ becomes to
\[
r_n < \frac{2^{-p_{n+1}}}{2}.
\]
Therefore is
\[
\theta(n) = 2^{-p_{n+1}} + r_n < \frac{3}{2} \cdot 2^{-p_{n+1}}
\]
and
\[
2^{p_{n+1}} < \frac{3}{2\theta(n)}
\]
and because of (\ref{G-5.4.0})
\[
p_{n+1} < \log_2\frac{3}{2\theta(n)} < \log_2\frac{2}{\theta(n)} < p_{n+1}+1
\]
respectively
\begin{equation}\label{G-5.5}
p_{n+1} = \left[\log_2\frac{3}{2\theta(n)}\right].
\end{equation}

%% file: KapOrder.tex
\section{The ''Order'' of the Primes}
Now it follows the question from the above, whether an order of the primes on the number line is confirmed by the \textit{formula of Gandhi}.
At first sight you would mean that a certain rule of the order of the primes on the number line would be visible, because the prime $p_{n+1}$ can be calculated by the primes $p_1,p_2, \ldots,p_n$.
But if we insert (\ref{G-5.1.1}) in (\ref{Ga-1.1}) then we get
\begin{equation}\label{Ga-6.1}
p_{n+1} = \left[1 - \log_2\left(\sum_{\indbox{t \geq p_{n+1}}{\CGD(t,\pz{p_n})=1}} 2^{-t}\right)\right].
\end{equation}
Here it becomes manifestly that the set of the summation indices
\[
Q = \left\lbrace t \in \M{N} \mid t \geq p_{n+1} \wedge \CGD(t, \pz{p_n})=1 \right\rbrace
\]
is a subset of the natural numbers which remain if all singles and multiples of the primes $p_1,p_2, \ldots,p_n$ from $\M{N}$ are sieved
\footnote{Sieve of Eratosthenes}.
And the least number in this subset is just the prime $p_{n+1}$
\footnote{see (\ref{G-5.4.1})}.
It is not ''calculated'' but the set of the summation indices is limited on the singles and multiples of the on $p_n$ subsequent primes. Therefore no order of the primes is confirmed but its formation on the number line is used as it is, without to disclose its secret.

Hence the question for order or disorder of the primes is not answerable with the formula of Gandhi.

%% file: Literatur.tex
%------------------------------------------------------------------